 \documentclass[12pt]{article}

\usepackage{amstext    }
\usepackage{amsthm    }
\usepackage{a4}
\usepackage[mathscr]{eucal}
\usepackage{mathrsfs}

\usepackage{amsmath}
\usepackage{amssymb}
\usepackage{amscd}

\newtheorem{theorem}{Theorem}[section]

\newtheorem{proposition}[theorem]{Proposition}
\newtheorem{corollary}[theorem]{Corollary}
\newtheorem{lemma}[theorem]{Lemma}

\newcommand{\cali}[1]{\mathscr{#1}}

\newcommand{\dist}{\mathop{\mathrm{dist}}\nolimits}

\newcommand{\ddc}{dd^c}

\newcommand{\id}{\mathop{\mathrm{id}}\nolimits}

\newcommand{\Alb}{{\rm Alb}}
\newcommand{\alb}{{\rm alb}}

\newcommand{\Cc}{\cali{C}}

\newcommand{\Kc}{\cali{K}}

\newcommand{\FS}{{\rm FS}}

\newcommand{\C}{\mathbb{C}}

\newcommand{\Z}{\mathbb{Z}}
\newcommand{\R}{\mathbb{R}}
\renewcommand{\P}{\mathbb{P}}
\newcommand{\D}{\mathbb{D}}

\newcommand{\zwedge}{\stackrel{\circ}{\wedge}}

\title{On the dynamical degrees of  meromorphic  maps preserving a fibration}

\author{Tien-Cuong Dinh, Vi{\^e}t-Anh Nguy{\^e}n and  Tuyen Trung Truong}

\begin{document}

\maketitle

\begin{abstract}
Let $f$ be a dominant meromorphic self-map on a  compact K{\"a}hler
manifold $X$ which preserves a meromorphic fibration $\pi:X\rightarrow
Y$ of $X$ over a  compact K{\"a}hler manifold $Y$. We compute
the dynamical degrees of $f$ in term of  its dynamical degrees relative to the
fibration and the dynamical degrees of the map $g:Y\rightarrow
Y$ induced by $f$.  We derive  from this  result   new properties of
     some fibrations intrinsically  associated  to $X$  when this  manifold  admits  an interesting  dynamical  system.
\end{abstract}

\noindent
{\bf Classification AMS 2010:} 37F, 14D (primary), 32U40, 32H50 (secondary).


\section{Introduction} \label{introduction}

Let $X$  be  a  compact K{\"a}hler  manifold of dimension
$k$ and $\omega_X$ a K\"ahler form on $X$.  Consider  a meromorphic self-map $f:X\rightarrow X$. Assume  that $f$ is {\it dominant,}  i.e.    the image of $f$ contains an open subset of
$X$.    The
iterate of order $n$ of $f$ is defined by $f^n:=f\circ\cdots\circ f$
($n$ times) on a dense Zariski open set and extends to a
dominant meromorphic map on $X$. 

Define, for $0\leq p\leq k$ and $n\geq 0,$
$$\lambda_p(f^n):=\|(f^n)^*(\omega_X^p)\|=\int_X
(f^n)^*(\omega_X^p)\wedge \omega_X^{k-p}.$$ 
It was shown in \cite{DinhSibony1, DinhSibony2} that the sequence
$[\lambda_p(f^n)]^{1/n}$ converges to a constant $d_p(f)$ which is 
the {\it dynamical degree of order $p$} of $f$.  
It
measures the growth of the  norms of  $(f^n)^*$
acting  on the Hodge cohomology group $H^{p,p}(X,\R)$ when $n$  tends
to infinity.  

Dynamical degrees $d_p(f)$ play    a central role  in the study of the   dynamical  system associated  to $f$, e.g. on the computation of entropies, the construction of invariant currents and the equidistribution problems. 
We refer  the reader to  
\cite{DinhSibony3,Gromov,Sibony, Yomdin} for more results on this  matter.

By the  mixed  version of Hodge-Riemann bilinear  relations \cite{DinhNguyen1, Gromov1, Khovanskii,Teissier, Timorin}, the dynamical
degrees of $f$ are log-concave, i.e. $p\mapsto \log d_p(f)$ is
concave or equivalently $d_p(f)^2\geq d_{p-1}(f)d_{p+1}(f)$ for $1\leq p\leq k-1$. 
Therefore, there are intergers $p\leq p'$ such that 
$$1=d_0(f)<\cdots<d_p(f)=\cdots=d_{p'}(f)> \cdots >d_k(f)\geq 1.$$

 An important problem in Complex Dynamics is to find dynamically
interesting examples of meromorphic self-maps on  compact K{\"a}hler manifolds.  We may rephrase   the  question in  a 
different way  by characterizing      manifolds $X$ on which there is  a    self-map $f$  with
distinct consecutive dynamical degrees, i.e. with $p=p',$  since  this
condition   prevents  the associated dynamical system  from  containing  neutral directions, e.g. $f=\id_Y\times g$ on $X=Y\times Z$ for some meromorphic self-map $g$ on $Z$.  

A meromorphic self-map $f:X\rightarrow X$ always   preserves  certain natural   meromorphic  fibrations  associated to $X$ that
we  encounter   in Algebraic  Geometry, see e.g. Amerik-Campana \cite{AmerikCampana} and Nakayama-Zhang
\cite{NakayamaZhang, Zhang1}.
These fibrations  are the   key tool     in the  classification theory of algebraic varieties and compact complex spaces, see e.g. Ueno's book \cite{Ueno}. So, in order to answer  the  above question  we are  led, in a natural way, to study
 self-maps which preserve fibrations.

  Let $\pi:X\rightarrow Y$ be a dominant meromorphic map from
$X$ onto a compact K{\"a}hler manifold $Y$ of dimension $l\leq
k$. This map defines a fibration on $X$ which might be
singular.  Suppose that $f$ preserves this fibration, i.e. $f$ sends generic
fibers of $\pi$ to fibers of $\pi.$ This property is equivalent to the existence of  a dominant meromorphic
map $g:Y\rightarrow Y$ such that  $\pi\circ f= g\circ \pi$. We say
that $f$ is {\it semi-conjugate} to $g$ or more precisely,  {\it $\pi$-semi-conjugate} to
$g$.    
In  this  context,  the  first and  second authors  have  introduced   in \cite{DinhNguyen} the 
{\it dynamical degree $d_p(f|\pi)$ of order $p$ of $f$     relative to $\pi$} for $0\leq  p\leq k-l.$ Roughly  speaking, this  quantity measures the
growth of $(f^n)^*$ acting on the subspace $H_\pi^{l+p,l+p}(X,\R)$
of classes in $H^{l+p,l+p}(X,\R)$ 
which can be supported by a generic fiber of $\pi$.
Precise   formulations  will be recalled    in Section    \ref{section_proofs} below. 
 
 The main purpose  of  the present  work is  to  quantify  the  relation between 
 the  dynamical  systems  associated   to semi-conjugate maps. In our  view, this  quantification, which is   formulated in terms of dynamical  degrees,
 has   at least  two    immediate  consequences. First,   it  will   shed  a light  to  the above question of characterizing   manifolds with interesting dynamical  systems. Second,  it   shows  that  the  dynamical  system of  a  map $f$  could be understood by studying the  dynamics of  a simpler map $g$ to which $f$ is semi-conjugate.
Here is our main result.

\begin{theorem} \label{th_main}
Let  $X$ and $Y$ be  compact K{\"a}hler manifolds of dimension $k$
and  $l$ respectively with $k\geq l$.
Let $f: X\rightarrow X$, $g:Y\rightarrow Y$ 
and   $\pi:X\rightarrow Y$ be dominant meromorphic maps 
such that $\pi\circ f=g\circ \pi$. 
Then, we have
       $$d_p(f)= \max_{\max\{0,p-k+l\}\leq j\leq \min\{p,l\}}d_j(g)d_{p-j}(f|\pi)$$
for every $0\leq p\leq k$.
\end{theorem}

Theorem  \ref{th_main} completes  the  work in   \cite{DinhNguyen}
where the  case  when   $X$ and $Y$ are projective  manifolds   has  been proved.
Note that the condition $\max\{0,p-k+l\}\leq j\leq \min\{p,l\}$ is
equivalent to $0\leq j\leq l$ and $0\leq p-j\leq k-l$  which insures that $d_j(g)$ and $d_{p-j}(f|\pi)$ are meaningful.  
We have the following useful consequence.

\begin{corollary} \label{cor_distinct_degree}
Let $f,\pi,g$ be as in Theorem \ref{th_main}. If the consecutive
dynamical degrees of $f$ are distinct, then the same property holds
for  $g$ and for the consecutive dynamical degrees of $f$ relative to
$\pi.$  
\end{corollary}

We  deduce  from  Corollary  \ref {cor_distinct_degree}  various  algebro-geometric properties of manifolds  admitting dynamically  interesting self-maps.
The following result is obtained using the Iitaka
fibrations of $X$.

\begin{corollary} \label{cor_kodaira}
Let $X$ be a  compact K{\"a}hler manifold admitting a dominant meromorphic self-map
with distinct consecutive dynamical degrees. Then, the Kodaira
dimension of $X$ is either equal to $0$ or $-\infty$.
\end{corollary} 

Note that the same result  was proved for compact K{\"a}hler surfaces 
by Cantat in \cite{Cantat} and Guedj in
\cite{Guedj},  for holomorphic maps on compact K{\"a}hler manifolds
by Nakayama and Zhang in \cite{NakayamaZhang,Zhang}, and  for
meromorphic  maps on projective manifolds  by  the first and  second authors  in \cite{DinhNguyen}. 

Applying Corollary \ref{cor_distinct_degree} to the Albanese fibration of $X$, we get the following result.

\begin{corollary} \label{cor_albanese}
Let $X$ be a  compact K{\"a}hler manifold admitting a dominant meromorphic self-map
with distinct consecutive dynamical degrees. Then, the  Albanese map $\alb:\ X\rightarrow\Alb(X)$ is surjective.
\end{corollary} 

In his  recent works \cite{Campana1},
  Campana  has constructed,   for an arbitrary compact K{\"a}hler manifold $X,$   its  core
fibration. This  fibration functorially  decomposes $X$ into  its  special components (the fibers) and  its general type component (orbifold base).  
In  the light of his  construction,  we  obtain the following result as a  consequence  of Corollary \ref{cor_distinct_degree}.

\begin{corollary} \label{cor_campana}
Let $X$ be a  compact K{\"a}hler manifold admitting a dominant meromorphic self-map
with distinct consecutive dynamical degrees. Then, $X$ is special  in the sense of Campana, see Definition  2.1 in \cite{Campana1} for the terminology.  
\end{corollary} 

We give  here  the outlines  of  the article.
The  main ingredient  in the proof of Theorem  \ref{th_main} is the introduction of an    approximation calculus   for positive
closed currents on compact K\"ahler manifolds which is  well-adapted  with  respect to a  holomorphic (possibly singular) fibration.
Using  the peculiarity of   projective  manifolds  
 a primitive  form of this calculus    has been  achieved  in \cite{DinhNguyen} where it has  played 
a key role  in the  proof of the main result, see   Proposition 2.3 and Proposition 2.4 therein. 

This  peculiarity  is  not available  any more in the context of general  compact K\"ahler manifolds.
Our new idea  here is  to reduce the above  calculus, via  a  well-chosen bi-meromorphic model, to   the problem
of approximating, with  mass control,   positive
closed currents defined  on submanifolds. This  is  the content of Section  \ref{section_current} below. 
Section    \ref{section_proofs} is  devoted to the proof of the main theorem  and its  corollaries.  Although we  adopt here   the  strategy
of \cite{DinhNguyen},  our  exposition is  somewhat  simpler and more constructive.

\medskip

\noindent
{\bf Acknowledgment.}
The third author would like to thank Professor Eric Bedford for introducing
the work in the paper \cite{DinhNguyen}, which initiated the
interest in this project. He also would like to thank Professors J\'anos
Koll\'ar, Jaroslaw Wlodarczyk, 
Valery Lunts and Sergey Pinchuk for helps. 
This paper was partially written during the visit of the first author at Humboldt Universit\"{a}t zu Berlin and the visit of the third author at University Paris-Sud. They would like to thank these organizations, the Alexander von Humboldt foundation, Professors J\"{u}rgen Leiterer and Nessim Sibony for their supports and their hospitality.

\section{Calculus on positive closed currents} \label{section_current}

In this section, we  develop  a delicate approximation  theory for positive
closed currents on compact K\"ahler manifolds.  This  is  the key  ingredient  for  our  method.
Note that by positive currents we means positive currents in the strong sense, see e.g. \cite[A.2]{DinhSibony3} for the terminology.
We will mostly apply our results to either currents of integration on varieties or {\it almost-smooth} currents, i.e. currents given by $L^1$-forms which are smooth outside a proper analytic subset.
We refer the reader to the books by Demailly \cite{Demailly} and Voisin 
\cite{Voisin} for the basic facts on currents and on K\"ahler geometry.

In what follows, if $T$ is a current and $\phi$  is a differential form on a manifold
$M$, both the pairings  $\langle T,\phi\rangle$ and $\langle \phi,T\rangle$ denote  the value of $T$ at $\phi$. In particular, when $T$ is also a differential form, these pairings are equal to the integral of $T\wedge\phi$ on $M$.  The cohomology class of a closed current is denoted by $\{\cdot\}$ and the current of integration on an analytic set is denoted by $[\cdot]$. 

Consider  a
compact K{\"a}hler manifold   $(X,\omega_X)$  of dimension $k$ as  above. Let $\Kc^p(X)$ denote the
cone of classes of smooth strictly positive closed $(p,p)$-forms in
$H^{p,p}(X,\R)$. This is an open salient cone,
i.e. $\overline{\Kc^p(X)}\cap -\overline{\Kc^p(X)}=\{0\}$. 
If $c,c'$ are two classes in $H^{p,p}(X,\R)$, we write $c\leq c'$ and
$c'\geq c$ when $c'-c$ is a class in $\overline{\Kc^p(X)}$. 

If $T,T'$ are two real currents of bidegree $(p,p)$, we write $T\geq T'$ and $T'\leq T$ when $T-T'$ is a positive current. If $T$ is a positive closed $(p,p)$-current, the {\it mass} of $T$ is defined by
$\|T\|:=\langle T,\omega_X^{k-p}\rangle$. This quantity is equivalent to the classical mass-norm of $T$ but it has the advantage that it depends only on the cohomology class $\{T\}$ of $T$.  

The following semi-regularization of currents 
was proved by Sibony and the first author in \cite{DinhSibony1,DinhSibony2}.

\begin{proposition} \label{prop_reg}
Let $T$ be a positive closed $(p,p)$-current on a compact K{\"a}hler
manifold $(X,\omega_X)$. Then, there is a sequence of smooth positive
closed $(p,p)$-forms $T_n$ on $X$ which converges weakly to a positive closed
$(p,p)$-current $T'\geq T$ such that $\|T_n\|\leq A\|T\|$ and
$\{T_n\}\leq A\|T\|\{\omega_X^p\}$, where $A>0$ is a constant  that depends only on $(X,\omega_X)$.
In particular, we have
   $\{T\}\leq A\|T\|\{\omega_X^p\}.$   
 Moreover, if $T$ is smooth on an open set $U$,
then for every compact set $K\subset U$, we have $T_n\geq T$ on $K$
when $n$ is large enough.  
\end{proposition} 

This semi-regularization of currents is the main technical tool in the proof of several results in Complex Dynamics. However, for the main results of this work, we will need refiner versions of the above proposition which somehow take into account the presence of a fibration on $X$. We will give now some statements in an abstract setting which may have an independent interest. 
The following result generalizes Proposition \ref{prop_reg}.

\begin{proposition}\label{prop_reg_intermediate}
Let $(X,\omega_X)$ be a compact K\"ahler manifold of dimension $k$ and  $W$ a submanifold of dimension $r$ 
of $X$. Let $\iota:W \hookrightarrow X$ denote the canonical embedding. Let $S$ be a positive closed $(p,p)$-current of mass $1$ on $W$. 
Then, there are smooth positive closed  $(p,p)$-currents $T_n$ on $X$ such that $\iota^*(T_n)$ converge to  a current $S'\geq S$ and that the masses of $T_n$ are bounded by a constant $c=c(X,\omega_X,W)$ which is independent of $S$. 
\end{proposition}
\proof
By Proposition \ref{prop_reg} applied to the current $S$ on $W$, it is enough to consider the case where $S$ is smooth, see also Lemma \ref{lem_double_indices} below.
Let $\pi_1,$ $\pi_2$ denote the canonical projections from $W\times X$ onto its factors. The idea of the proof is to write 
$$\iota_*(S)=(\pi_2)_*(\pi_1^*(S)\wedge [\Delta_W]).$$
Then, in order to obtain $T_n$, we have just to replace $[\Delta_W]$ by a suitable smooth positive closed current on $W\times X$.

Let  $\Pi:\widehat{W\times X} \to W\times X$ be  the  blow-up of  $W\times X$ along the diagonal $\Delta_W$ of $W\times W$. 
By Blanchard's theorem \cite{Blanchard}, $\widehat{W\times X}$ is a K\"ahler manifold.
Fix a large enough K\"{a}hler form $\omega$ on $\widehat{W\times X}.$ Consider  $U:=\Pi_*(\omega).$ Then, $U$ is a positive closed $(1,1)$-current on $W\times X$ which is smooth outside the diagonal 
$\Delta_W$ and have Lelong number $\geq 1$ at each point of $\Delta_W$. Adding to $\omega$ the pull-back of a K\"ahler form on $W\times X$ allows to assume that $U$ belongs to a K\"ahler class.

Fix a K\"ahler form $\alpha$ in the cohomology class $\{U\}$. There is a  quasi-p.s.h. function $u$ on $W\times X$ such that
$U=\alpha+\ddc u.$  Such  a function is  called  a {\it quasi-potential} of $U.$
We claim that there is   a sequence of smooth positive closed $(1,1)$-forms $U_n$ with decreasing smooth quasi-potentials $u_n$ such that
 $\lim_{n\to\infty} U_n=  U, $ that is, $U_n  =\alpha+\ddc u_n$ and  $u_n\searrow u$ as $n\nearrow\infty.$
Indeed, it is enough to take $U_n:=\alpha+\ddc u_n$, where we define $u_n$ as $\max_{\epsilon_n}(u,-n)$ for a suitable regularization $\max_{\epsilon_n}(\cdot,\cdot)$ of the function $\max(\cdot,\cdot)$. 

Define
$$T_n:=(\pi_2)_*\big(\pi_1^*(S)\wedge U_n^r\big).$$
Since $U_n$ is smooth and the $\pi_i$ are submersions, $T_n$ is also smooth.  All currents we consider have masses bounded by a constant  $c=c(X,\omega_X,W)$ because their cohomology classes are bounded. Extracting a subsequence allows us to assume that $\iota^*(T_n)$ converge to a current $S'$. 
To complete the proof it suffices to check that $S'\geq S.$

Let $\Phi$ be an arbitrary {\it weakly} positive test form of  bidegree $(r-p,r-p)$ on $W$, see e.g. \cite[A.2]{DinhSibony3} for the terminology.   We  need to check that
$\langle S',\Phi\rangle\geq   \langle S,\Phi\rangle.$
Let  $\tau_1,$ $\tau_2$ denote  the  canonical projections from $W\times W$ onto its factors.   
Consider the diagram 
$$ W\times W\stackrel{ \id_W\times \iota}{\hookrightarrow}W\times X.$$
We have for $S$ smooth
\begin{eqnarray*}
\langle \iota^*(T_n),\Phi\rangle
 &=&\big\langle (\pi_2)_*\big (\pi_1^*(S)\wedge   U_n^r  \big ),\iota_*\Phi\big\rangle\\
&=&\big\langle S,  (\pi_1)_*\big( U_n^r  \wedge\pi_2^*\iota_*\Phi \big)  \big\rangle\\
&=&\big\langle S,  (\tau_1)_* \big ( (\id_W\times \iota)^*U_n^r  \wedge \tau_2^*\Phi\big )   \big\rangle.
\end{eqnarray*}

Note that the above identities hold also for smooth currents $S$ which are not positive closed. Moreover, the first and the last integrals are meaningful for all $S$ of order 0 and depend continuously on $S$. Thus, by continuity, these integrals are also equal
for $S$ as in our proposition. It follows that
$$\langle S',\Phi\rangle =\lim\limits_{n\to\infty} \langle \iota^*(T_n),\Phi\rangle
 =\lim\limits_{n\to\infty} \langle S,  (\tau_1)_* \big ( (\id_W\times \iota)^*U_n^r  \wedge \tau_2^*\Phi\big )   \rangle.$$
Therefore, in order  to show that the last integral is greater than  $\langle S,\Phi\rangle$ it suffices  to check that  any limit value
 of  the sequence  $ (\id_W\times \iota )^*U_n^r$ is larger than or equal to
$[\Delta_W].$     

To obtain this inequality, we suppose  without loss of generality that the sequence   $ (\id_W\times \iota )^*U_n^r$
converges weakly to a current $U'.$ Clearly,
\begin{eqnarray*}
 (\id_W\times \iota )_*U'&=& \lim\limits_{n\to\infty}(\id_W\times \iota )_*(\id_W\times \iota )^* U^r_n\\
 & = &  \lim\limits_{n\to\infty}U^r_n\wedge [W\times W]\\
&=&  U^r\wedge [W\times W],
\end{eqnarray*}
because the $U_n$ admit quasi-potentials which decrease to a quasi-potential of $U$. 
Note that the last wedge-product is well-defined since $U$ is smooth outside $\Delta_W$ and $\dim \Delta_W=r$, see e.g. Demailly \cite{Demailly}. 

Hence,   we only need to show  that  $U^r\wedge [W\times W]\geq [\Delta_W].$
But this can be  checked  using a local model of the blow-up  $\Pi:\widehat{W\times X} \to W\times X.$ Indeed, 
consider a $(k+r)$-dimensional polydisc $\D$ in $W\times X$  with holomorphic coordinates  $z_1,\ldots,z_{k+r}$    around  an arbitrary fixed point  in $W\times W.$ We can choose these local coordinates so that 
$W\times W$ is equal to the linear subspace $\{z_{r+1}=\cdots=z_{k}=0\}$ and $\Delta_W$ is equal to the linear subspace $\{z_1=\cdots=z_k=0\}$.
Let $[w_1 : \cdots : w_k]$  be the homogeneous coordinates on $\P^{k-1}.$ Then,  $\widehat{W\times X}\cap \Pi^{-1}(\D)$ may be
identified  with  the  smooth manifold
$$
\widehat{\D}:=\left\lbrace  (z_1,\ldots,z_{k+r},[w_1:\cdots:w_k])\in\D\times \P^{k-1}:\   z_iw_j=z_jw_i   
\text{ for } 1\leq i,j\leq k\right\rbrace.
$$

Observe that   $\Pi$ is induced by the canonical projection from  $\D\times\P^{k-1}$ onto the factor
$\D$.  Let $\Pi'$ be the canonical projection from  $\widehat{\D}$ onto the factor $\P^{k-1}.$ Let    $\omega_\FS$  be the Fubini-Study form on $\P^{k-1}.$ 
Recall that $\omega_\FS$ is induced by the $(1,1)$-form $\ddc\log\|(w_1,\ldots,w_k)\|$ on $\C^k\setminus\{0\}$.
Since  $\omega$ is large enough, we have $\omega\geq {\Pi'}^*(\omega_\FS)$. Since  $[w_1:\cdots:w_k]=[z_1:\cdots: z_k]$ outside the exceptional hypersurface $\Pi^{-1}(\Delta_W)$ of $\widehat D$,  we obtain 
 $$\Pi_*(\omega)\geq   \Pi_* {\Pi'}^*(\omega_\FS)      = \ddc \log\|( z_1,\ldots, z_k)  \|.$$
 This inequality holds on $\D\setminus\Delta_W$ and hence on $\D$ since positive closed $(1,1)$-currents have no mass on subvarieties of codimension $\geq 2$. 
 
Finally, we deduce that
 \begin{eqnarray*}
 U^r\wedge  [W\times W] & = &   \Pi_*(\omega)^r \wedge  [W\times W] \\
 &\geq &   (\ddc \log\| (z_1,\ldots, z_k)  \|)^r \wedge  [W\times W]\\
&=&   (\ddc \log\| (z_1,\ldots, z_r)  \|)^r \wedge  [W\times W].
 \end{eqnarray*}
The last current is equal to $[\Delta_W]$. So, the proof of
the lemma is completed.
\endproof

Before giving the main result in this section, let us introduce some useful notions that we will need in our computation.
Let $(M,\omega_M)$  be  a  compact K\"ahler manifold  of dimension $m$.
In \cite{DinhSibony4}  Sibony  and the first  author  have introduced  the  following natural metric on the space of positive closed $(p,p)$-currents on $M$.   If $R$ and $S$ are such currents,
define
$$
\dist(R,S) := \sup\limits_{\| \Phi\|_{\Cc^1}\leq 1}
|\langle R-S,\Phi\rangle|,
$$
where $\Phi$ is a smooth $(m-p, m-p)$-form on $M$ and we use the
sum of $\Cc^1$-norms of its coefficients for a fixed atlas on $M$.  Recall  the  following result  from Proposition 2.1.4 in \cite{DinhSibony4}.

\begin{lemma}\label{lem_double_indices}
On the convex set of positive closed $(p,p)$-currents of mass $\leq 1$ on $M$, the topology induced by the above distance coincides with the weak
topology.
\end{lemma}

Consider now a dominant meromorphic map  $h:\ (M,\omega_M)\rightarrow (N,\omega_N)$ between compact K{\"a}hler manifolds.   It is  well-known (see e.g. \cite{DinhNguyen, DinhSibony1,
  DinhSibony2}) that
$h$ induces  the linear operators  $h^*$ and $h_*$ acting on smooth forms. 
  In general, the above operators do not extend continuously to positive
closed currents. We will use instead the strict pull-back of currents
$h^\bullet$   which coincides with $h^*$ 
  on smooth positive closed forms.

Let $U$ be the maximal Zariski open
set in $M$ such that $h:\ U\rightarrow h(U)$ is locally  a submersion.
The complement of $U$ in $M$ is called the {\it critical set} of $h$.
If $T$ is a positive closed $(p,p)$-current on $N$, $(h_{|U})^*(T)$
is well-defined and is a positive closed $(p,p)$-current on
$U$. Proposition \ref{prop_reg} allows us  to show that this current has finite
mass. By Skoda's  theorem \cite{Skoda}, its trivial extension to $M$
is a positive closed $(p,p)$-current that we denote by
$h^\bullet(T)$. 
 We will use the property that 
\begin{equation}\label{eq_property_strict_pull_back}
\| h^\bullet(T)\|\leq A\|T\|
\end{equation} 
for some constant $A>0$ independent of $T$, see \cite{DinhSibony1,
  DinhSibony2} for details.

Let $T$ and $S$ be positive closed currents on $M$ of bidegrees $(p,p)$ and $(q,q)$ respectively with $p+q\leq m$. 
Assume
that $T$ is smooth on  a
dense Zariski open set $U$ of $M$.
Then, $T_{|U}\wedge
S_{|U}$ is well-defined and has a finite mass.  Therefore, by Skoda's theorem
\cite{Skoda}, its trivial extension defines a positive closed current on
$M$. We denote by $T\zwedge S$ this current obtained for the maximal
Zariski open set $U$ on which $T$ is smooth (in that case $T_{|U}$ is the regular part of $T$). 
Observe that when $S$ has no mass on proper analytic subsets of $M$, the
current obtained in this way does not change if we replace $U$ with a smaller dense
Zariski open set. 
We have  the following result, see Lemma 2.2 in \cite {DinhNguyen}.

\begin{lemma} \label{lemma_wedge_zariski}
There is a constant  
$A>0$ independent of $T$ and $S$ such that
 $$\|T\zwedge S\|\leq A\|T\|\|S\|.$$
\end{lemma}

We now state the main result of this section. It is the key technical tool in our proof of the main theorem.
Let $\pi:(X,\omega_X)\rightarrow (Y,\omega_Y)$ be a dominant
holomorphic map between compact K{\"a}hler manifolds of dimension $k$
and $l$ respectively with $k\geq l.$  Let $T$ be a positive closed $(p,p)$-current on $X$. Define for $\max \{0,p-k+l\}\leq j\leq \min \{l,p\}$, or equivalently, for $0\leq j\leq l$ and $0\leq p-j\leq
k-l$,
\begin{equation} \label{eqn_alpha}
\alpha_j(T):=\big\langle T,\pi^*(\omega_Y^{l-j})\wedge
\omega_X^{k-l-p+j}\big\rangle.
\end{equation}
Observe that $\alpha_j(T)$ depends only on the cohomology class $\{T\}$ of $T$. Moreover, if $A$ is a constant such that $\pi^*(\omega_Y)\leq A\omega_X$, then $\alpha_j(T)\leq A\alpha_{j+1}(T)$. 

Denote by $\smile$ the cup-product on the Hodge cohomology ring. 
The following result holds for a larger class of currents $T$ but for simplicity we limit ourself in the case that we need.

\begin{proposition} \label{prop_reg_bis}
 Let $T$ be  an almost-smooth  positive closed $(p,p)$-current  on $X$. Then, there are  positive closed smooth $(p,p)$-forms $T_n$ on $X$  converging  to a  positive closed current $T'\geq T$
such that
 \begin{equation*}
 \{T_n\}\leq A\sum _{\max \{0,p-k+l\}\leq j\leq \min \{l,p\}}\alpha _j(T)\{\pi^*(\omega _Y^j)\}\smile \{\omega _X^{p-j}\},
\end{equation*}
where $A>0$ is a constant  that depends only on $(X,\omega_X)$. In particular, we have
\begin{equation*}
\{T\}\leq A\sum _{\max \{0,p-k+l\}\leq j\leq \min \{l,p\}}\alpha _j(T)\{\pi ^*(\omega _Y^j)\}\smile \{\omega _X^{p-j}\}
\end{equation*}
and 
$$\alpha_j(T_n)\leq A\alpha_j(T)$$
for some constant $A>0$ that depends only on $(X,\omega_X)$.
 \end{proposition}

Recall that $\alpha_j(\cdot)$ is bounded by a constant times $\alpha_{j+1}(\cdot)$. We also have
$\omega_Y^{l+1}=0$ since $\dim Y=l$. 
So, from the definition of $\alpha_j(\cdot)$, it is not difficult to see that 
the last assertion of Proposition \ref{prop_reg_bis} is a direct consequence of the first one. 
The rest of this section is  devoted  to the proof of  the first assertion of that proposition.
For this  purpose we  need some preparatory results.

Let $\pi _1,\pi _2: X\times X\rightarrow X$ be the canonical projections onto the first and second  factors. Denote by $\Delta _X$ and $\Delta_Y$ the diagonals of $X\times X$ and of $Y\times Y$ respectively.
Then, $(\pi \times \pi )^{-1}(\Delta _Y)$ is an analytic subvariety of $X\times X$ which contains $\Delta _X$.

\begin{lemma}\label{LemmaVarietyContainingDeltaX}
There is a unique irreducible component $V$ of $(\pi \times \pi )^{-1}(\Delta _Y)$ which contains $\Delta _X$. Moreover, $V$ has dimension
$2k-l$ and the singular locus of $V$ does not contain $\Delta _X$.
\end{lemma}
\proof
Let $Z$ denote the set of critical values of $\pi$. By Bertini-Sard theorem, $Z$ is a proper analytic subset of $Y$. Define $Y':=Y\setminus Z$,
$X':=X\setminus \pi ^{-1}(Z)$ and $\pi':=\pi_{|X'}$. So, $\pi':X'\to Y'$ is a submersion and
$(\pi'\times \pi')^{-1}(\Delta _{Y'})$ is a smooth complex submanifold of dimension $2k-l$ of $X'\times X'$.
Note that $(\pi '\times \pi ')^{-1}(\Delta _{Y'})$ is the trace of $(\pi \times \pi )^{-1}(\Delta _Y)$ in  $X'\times X'$.
Hence, the regular part of $(\pi \times \pi)^{-1}(\Delta _Y)$ contains $(\pi '\times \pi ')^{-1}(\Delta _{Y'})$. 

Since $\Delta _{X'}$ is irreducible and is contained in $(\pi '\times \pi ')^{-1}(\Delta _{Y'})$ which is a smooth manifold, there is a unique irreducible component $V'$ of that manifold which contains $\Delta_{X'}$. Finally, since $\Delta_{X'}$ is a dense Zariski open set of $\Delta_X$, the unique irreducible component $V$ of $(\pi\times\pi)^{-1}(\Delta_Y)$ containing $V'$ is also the unique component which contains $\Delta_X$.  Its dimension is equal to $\dim V'=2k-l$. Its singular locus does not contain $\Delta_X$ since its regular part contains $\Delta_{X'}$.
\endproof

By Lemma \ref{LemmaVarietyContainingDeltaX} and the embedded resolution theorem of Hironaka (see \cite{ArokaHironakaVicente} or Theorem 2.0.2 in  \cite{Wlodarczyk}), there is a finite
composition of blow-ups along smooth centers $\sigma :\widetilde{X\times X}\rightarrow X\times X$ with the following properties:
\begin{itemize}
\item [$\bullet$] If $E$ is the exceptional divisor of $\sigma$, then $\sigma (E)$ is contained in the singular locus of $V$, and thus is contained in $(X\times\pi ^{-1}(Z))\cup (\pi ^{-1}(Z)\times X)$.
\item [$\bullet$] The strict transform $\widetilde V$ of $V$ is smooth, and hence is a compact K\"{a}hler submanifold of $\widetilde{X\times X}$. We denote by
$\iota : \widetilde V \hookrightarrow \widetilde{X\times X}$ the inclusion.
\end{itemize}

Because $\Delta_X$ is not contained in $(X\times \pi ^{-1}(Z))\cup (\pi ^{-1}(Z)\times X)$,  its strict transform $\widetilde{\Delta}_X$ is
well-defined. This is an irreducible subvariety of dimension $k$ of $\widetilde V$ and its image by $\sigma$ is
$\Delta _X$. Since $\Delta _X$ is also irreducible, we get that
\begin{eqnarray*}
\sigma _*([\widetilde{\Delta}_X])=[\Delta _X].
\end{eqnarray*}

\begin{lemma}\label{LemmaCohomologyClassOfDeltaXStronger}
There are  smooth positive closed $(l,l)$-forms $\Delta _{Y,n}$ on
$Y\times Y$ and  smooth positive closed  $(k-l,k-l)$-forms $\Delta _{X,n}$ on $X\times X$, all with uniformly bounded masses, such that the limit $\Theta:=\lim _{n\rightarrow\infty} (\pi \times \pi )^*(\Delta _{Y,n})\wedge \Delta _{X,n}$ exists and is larger or equal to
$[\Delta_X]$.
\end{lemma}
\begin{proof}
Denote by $\Omega$ the $(k-l,k-l)$-current on $\widetilde V$ defined as the integration on $\widetilde{\Delta}_X$. Then,
$\iota _{*}(\Omega)=[\widetilde{\Delta}_X]$
as currents on $\widetilde{X\times X}$.
We apply Proposition  \ref{prop_reg_intermediate} for $\widetilde{X\times X}$ instead of $X$, 
for $W:= \widetilde{V}$,  $S:= \Omega$ and   $p:=k-l.$ Consequently,    there is a sequence of smooth positive closed $(k-l,k-l)$-forms $\Omega _r$ on $\widetilde{X\times X}$
with $\|\Omega_r\|$ uniformly bounded such that
$\lim _{r\rightarrow \infty}\iota ^*(\Omega_r)\geq \Omega.$
Hence,
\begin{eqnarray*}
\iota _*(\Omega)\leq \lim _{r\rightarrow\infty}[\widetilde{V}]\wedge \Omega_r.
\end{eqnarray*}

From the definition of $V$, we have $[\widetilde V]\leq ((\pi \times \pi)\circ \sigma )^{\bullet}[\Delta _Y]$. By
Proposition \ref{prop_reg}, there are  smooth positive closed $(l,l)$-forms $\Delta _{s,Y}$ on $Y\times Y$ with uniformly bounded masses such that $\lim _{s\rightarrow \infty}\Delta _{s,Y}\geq [\Delta _Y]$. 
 Hence,
\begin{eqnarray*}
[\widetilde{V}]\leq \lim _{s\rightarrow\infty} ((\pi \times \pi )\circ \sigma )^*(\Delta _{s,Y}).
\end{eqnarray*}
Here, in order to get the existence of the above limit,  we extract a subsequence if necessary. It follows that   
\begin{eqnarray*}
\iota _*(\Omega)\leq \lim _{r\rightarrow\infty}\lim _{s\rightarrow\infty} ((\pi \times \pi )\circ \sigma )^*(\Delta _{s,Y})\wedge \Omega_r.
\end{eqnarray*}
Since $\sigma _*(\iota _*(\Omega))=[\Delta _X]$, applying the projection formula gives
\begin{eqnarray*}
[\Delta _X]\leq \lim _{r\rightarrow\infty}\lim _{s\rightarrow\infty} (\pi \times \pi )^*(\Delta _{s,Y})\wedge \sigma _*(\Omega_r).
\end{eqnarray*}

Recall that $\|\Omega_r\|$ is bounded uniformly on $r$. Therefore,
by Proposition \ref{prop_reg}, there are  smooth positive closed $(k-l,k-l)$-forms $\Delta
_{r,t,X}$ on $X\times X$ with uniformly bounded masses such that
$\lim _{t\rightarrow \infty} \Delta _{r,t,X}\geq \sigma _*(\Omega_r)$.
Putting the above inequalities together, we obtain
$$\lim _{r\rightarrow\infty}\lim _{s\rightarrow\infty}\lim _{t\rightarrow\infty}(\pi \times \pi )^*(\Delta _{s,Y})
\wedge \Delta _{r,t,X}\geq [\Delta_X].$$
Since    $\|\Delta _{s,Y}\|,$   $\|\Delta _{r,t,X}\|$ are  uniformly bounded,  by Lemma  \ref{lem_double_indices},
we can extract   two sequences   $\Delta _{X,n}:=\Delta _{r_n,t_n,X}$ and    $\Delta _{Y,n}:=\Delta _{s_n,Y}$ such that
\begin{equation*}
\lim _{n\rightarrow\infty}  (\pi \times \pi ) ^*(\Delta _{Y,n})\wedge \Delta _{X,n} \geq [\Delta_X] .
\end{equation*}
This completes the proof. 
\end{proof}

\noindent{\bf  End of the proof of Proposition  \ref{prop_reg_bis}.}
Let  $\Delta_{Y,n}$ and $\Delta_{X,n}$ be  smooth positive closed forms given by Lemma \ref{LemmaCohomologyClassOfDeltaXStronger}.
Define
\begin{equation*}
T_n:=(\pi_1)_*\left\lbrack (\pi \times \pi )^*(\Delta_{Y,n})\wedge \Delta_{X,n}\wedge \pi_2^*(T)\right\rbrack.
\end{equation*}
So,  the $T_n$  are smooth positive  closed  $(p,p)$-forms on $X$ with  uniformly bounded masses. Hence, by  extracting a subsequence if necessary, we can  assume without loss of generality that the limit $T':=\lim_{n\to\infty}T_n$ exists.

Let $C$ be a proper analytic subset of $X$ so that $T$ is smooth on  $X\setminus C$. Define $U:=\pi_2^{-1}(X\setminus C)$.   By Lemma \ref{LemmaCohomologyClassOfDeltaXStronger},  since $T$ is smooth outside $C$, we get  easily that  
 \begin{equation*}
 [\Delta_X]_{|U}\wedge \pi_2^*(T)_{|U}
\leq\lim_{n\rightarrow\infty}  \left[(\pi \times \pi )^*(\Delta_{Y,n})\wedge \Delta_{X,n}\wedge \pi_2^*(T)\right]_{|U}.
\end{equation*}
It follows that 
$$T'=\lim_{n\to\infty} T_n \geq (\pi_1)_*([\Delta _X]_{|U}\wedge \pi _2^*(T)_{|U})=T$$
since the last current is almost-smooth and hence has no mass on proper analytic subsets of $X$. 

Now,  we turn to   the proof of  the first inequality in the proposition.
Let $\tau_1,\tau_2$ denote the projections from $Y\times Y$ onto its factors.
Define two K\"ahler forms on $Y\times Y$ and $X\times X$ by 
 $$
\omega_{Y\times Y}:=\tau_1^*(\omega _Y)+\tau_2^*(\omega _Y)\quad\text{and}\quad
\omega_{X\times X}:=\pi _1^*(\omega _X)+\pi _2^*(\omega _X).
$$
Since $\|\Delta _{Y,n}\|$ and $\|\Delta _{X,n}\|$ are uniformly bounded,   by Proposition \ref{prop_reg}, there is a constant
$A_1>0$ independent of $ n$ so that
$$\{\Delta _{Y,n}\} \leq A_1\{\omega _{Y\times Y}^l\} \quad \text{and} \quad
\{\Delta _{X,n}\} \leq A_1\{\omega _{X\times X}^{k-l}\}.
$$

Hence, we obtain
$$\{(\pi \times \pi )^*(\Delta _{Y,n})\} \leq A_1\{(\pi \times \pi )^*(\omega _{Y\times Y}^l)\}=A_1\{(\pi _1^*\pi ^*\omega _Y+\pi _2^*\pi ^*\omega _Y)^l\}$$
and
$$\{\Delta _{X,n}\}\leq A_1\{\omega _{X\times X}^{k-l}\}=A_1\{(\pi _1^*\omega _X+\pi _2^*\omega _X )^{k-l}\}.
$$
It follows that  $\{T_n\}$ is bounded from above by $A_1^2$ times  the class of 
$$S:=(\pi_1)_*\big[(\pi _1^*\pi ^*\omega_Y+\pi _2^*\pi ^*\omega _Y)^l\wedge (\pi _1^*\omega _X+\pi _2^*\omega _X)^{k-l}  \wedge \pi _2^*(T)\big].$$

Observe that $S$ is a linear combination of the forms
$\pi^*(\omega_Y^j)\wedge \omega_X^{p-j}$
with $\max \{0,p-k+l\}\leq j\leq \min \{l,p\}$. Moreover, the coefficient of
$\pi^*(\omega_Y^j)\wedge \omega_X^{p-j}$ in $S$ is equal to the following constant function, i.e. closed $(0,0)$-current, 
$$(\pi _1)_*\big[\pi _2^*(T) \wedge \pi _2^*(\omega _X^{k-l-p+j})\wedge \pi
_2^*\pi ^*(\omega _Y^{l-j})\big]
 =   (\pi _1)_*\pi _2^*\big[T \wedge \omega _X^{k-l-p+j}\wedge \pi ^*(\omega
_Y^{l-j}) \big].$$
So, it is equal to the mass of the measure
$$T \wedge \omega _X^{k-l-p+j}\wedge \pi ^*(\omega
_Y^{l-j}).$$
Therefore, we have
$$S =\sum _{\max \{0,p-k+l\}\leq j\leq \min \{l,p\}} \alpha _j(T)\pi ^*(\omega _Y^j)\wedge
\omega _X^{p-j}.$$
The proposition follows.
\hfill $\square$

 \section{Proof of the main results} \label{section_proofs}

Let us  start with the proof of Theorem  \ref{th_main}.
Although we follow closely  the  strategy for the main theorem  in  \cite{DinhNguyen},    our present exposition
is  simpler and more instructive thanks to the  results  of Section  \ref{section_current}.  
For the sake of completeness  and  for the  reader's convenience  we   give here the detailed proof.   

First, we  recall from Section 3  in   \cite{DinhNguyen} that the relative dynamical degrees are bi-meromorphic invariants. So, we can assume without loss of generality that $\pi$ is a holomorphic map. Recall also that the relative dynamical degree $d_p(f|\pi)$ of order $p$, with $0\leq p\leq k-l$, is  defined by 
\begin{equation*} 
d_p(f|\pi):=\lim_{n\to\infty}[\lambda_p(f^n|\pi)]^{1/n},
\end{equation*}
where
\begin{equation*}
\lambda_p(f^n|\pi)   :=\|  (f^n)^*
 (\omega_X^p)\wedge   \pi^*(\omega_Y^l)       \|     
  = \big\langle (f^n)^*
 (\omega_X^p)\wedge   \pi^*(\omega_Y^l),
\omega_X^{k-l-p}\big\rangle.
\end{equation*}
The reader will find in \cite{DinhNguyen} the geometric interpretation of these degrees.
  
Our calculus involves the following auxiliary quantities.
For $n\geq 0$ and $\max\{0,p-l\}\leq q\leq \min\{p,k-l\}$, define
\begin{equation*}
a_{q,p}(n)   :=\|  (f^n)^*
 (\omega_X^p)\wedge   \pi^*(\omega_Y^{l-p+q})       \|     
  = \big\langle (f^n)^*
 (\omega_X^p)\wedge   \pi^*(\omega_Y^{l-p+q}),
\omega_X^{k-l-q}\big\rangle.
\end{equation*}
Observe that 
\begin{equation}\label{eq_a_pp}
a_{p,p}(n)=  \lambda_p (f^n|\pi).
\end{equation}

Define also for $0\leq p\leq k$
$$b_p(n):=\sum_{\max\{0,p-l\}\leq q\leq \min\{p,k-l\}} a_{q,p}(n).$$
The following lemma shows that $b_p(n)$ is equivalent to $\lambda_p(f^n)$ when $n$ goes to infinity.

\begin{lemma} \label{lemma_b_p}
The sequence $b_p(n)^{1/n}$ converges to $d_p(f)$.
\end{lemma}
\proof
Since $\pi^*(\omega_Y^{l-p+q})$
 is smooth on $X$, we have 
$$a_{q,p}(n)= \big\langle (f^n)^*
(\omega_X^p)\wedge   \pi^*(\omega_Y^{l-p+q}),
\omega_X^{k-l-q})\big\rangle \leq
A \|(f^n)^*(\omega_X^p)\|=A\lambda_p(f^n)$$
for some constant $A>0$. We deduce that 
$$\limsup_{n\to\infty} b_p(n)^{1/n}\leq
d_p(f).$$ 
So, in order to obtain the lemma, it is enough to check that $\lambda_p(f^n)\leq Ab_p(n)$ for
some constant $A>0$. 

Applying Proposition  \ref{prop_reg_bis} to  $\omega_X^{k-p}$ gives
\begin{eqnarray*}
\{\omega_X^{k-p}\} &\leq & A\sum _{\max \{0,p-l\}\leq q\leq \min \{p,k-l\}}\{\pi ^*(\omega _Y^{l-p+q})\}\smile \{\omega _X^{k-l-q}\}
\end{eqnarray*}
for some constant $A>0$. 
This, combined  with the fact that $(f^n)^*
(\omega_X^p)$ is  positive closed, implies that
\begin{eqnarray*}
\lambda_p(f^n) & =&   \big\langle
(f^n)^*(\omega_X^p),\omega_X^{k-p}  \big\rangle\\
&\leq & A\sum_{\max\{0,p-l\}\leq q \leq \min\{p,k-l\}}    
 \big\langle  (f^n)^*(\omega_X^p), 
\pi ^*(\omega_Y^{l-p+q})\wedge
\omega_X^{k-l-q} \big\rangle \\
& \leq & A  \sum_{\max\{0,p-l\}\leq q \leq \min\{p,k-l\}} a_{q,p}(n).
\end{eqnarray*}
The last sum is equal to $b_p(n)$. The lemma follows.
\endproof

For $n\geq 0$ and $0\leq p\leq l$, define  
$$c_p(n) := \lambda_p(g^n)=\|(g^n)^{*}(\omega_Y^p)\|=\big\langle
(g^n)^{*}(\omega_Y^p), \omega_Y^{l-p}\big\rangle.$$
We have the following lemmas.

\begin{lemma} \label{lemma_Pi_wedge}
There is a constant $A>0$ such that
$$\Big\langle 
(f^n)^*\big( \pi^*\omega_Y^{p-q}\wedge\omega_X^q\big),\pi^*(\omega_Y^{l-p+p_0})\wedge
\omega_X^{k-l-p_0}\Big\rangle \leq A a_{p_0,q}(n)c_{p-q}(n)$$
for $0\leq p_0\leq k-l$, $p_0\leq p\leq l+p_0$, $p_0\leq q\leq p$ and $n\geq
0$.
Moreover, the above integral vanishes when $q<p_0$.  
\end{lemma}
\proof
We prove the first assertion.
Observe that 
\begin{eqnarray*}
(f^n)^*(\pi^*\omega_Y^{p-q}\wedge\omega_X^q) 
 =   (f^n)^*(\pi^*\omega_Y^{p-q})\zwedge
(f^n)^*(\omega_X^q).
\end{eqnarray*}
Hence, the left hand side of the inequality in the lemma is equal to
$$\big\langle 
(f^n)^*\pi^*\omega_Y^{p-q}\zwedge
(f^n)^*(\omega_X^q),\pi^*(\omega_Y^{l-p+p_0})\wedge
\omega_X^{k-l-p_0}\big\rangle.$$

Define 
$$T:=(f^n)^*\pi^*\omega_Y^{p-q}\wedge \pi^*(\omega_Y^{l-p+p_0}) \quad
\text{and} \quad S:=(f^n)^*(\omega_X^q)\wedge\omega_X^{k-l-p_0}.$$ 
These currents are of bidegree $(l-q+p_0,l-q+p_0)$ and
$(k-l+q-p_0,k-l+q-p_0)$ respectively.  They are almost-smooth and hence have no mass on proper analytic subsets. 
The left hand side of the inequality in the lemma is equal to the mass of the measure
$T\zwedge S$. 
Since
$\pi\circ f^n=g^n\circ\pi$, we have
$$T= (f^n)^*\pi^*(\omega_Y^{p-q})\wedge \pi^*(\omega_Y^{l-p+p_0})
= \pi^\bullet(g^n)^*(\omega_Y^{p-q})\wedge\pi^*( \omega_Y^{l-p+p_0}).$$
 
By Proposition  \ref{prop_reg}, for every fixed  $n$, there exist smooth positive closed forms $\beta_j$ of bidegree $(p-q,p-q)$ on $Y$ so that
\begin{itemize}
\item[$\bullet$]  $\|\beta _j\|\leq A\|(g^n)^*(\omega _Y^{p-q})\|=Ac_{p-q}(n)$ for all $j;$
\item[$\bullet$] $\lim _{j\rightarrow \infty}\beta_j\geq (g^n)^*(\omega _Y^{p-q}),$
\end{itemize}
 where $A>0$ is a constant  that depends only on $Y.$
Then, using  (\ref{eq_property_strict_pull_back}), we deduce from the above discussion that 
$$T\leq \lim _{j\rightarrow\infty}\pi ^*(\beta_j)\wedge \pi ^*(\omega _Y^{l-p+p_0})
=\lim _{j\rightarrow \infty}\pi ^*(\beta_j\wedge \omega _Y^{l-p+p_0}).
$$
Hence, since $T$ and $S$ are almost-smooth, we obtain
\begin{eqnarray*}
\|T\zwedge S\|\leq \lim _{j\rightarrow \infty} \big\langle   \pi^*(\beta_j\wedge \omega _Y^{l-p+p_0}),S\big\rangle.
\end{eqnarray*}
Since $\pi ^*(\beta_j\wedge \omega _Y^{l-p+p_0})$ are smooth, the right hand side  of the above inequality increases when we replace $\beta_j$ by any closed
smooth form having 	a larger cohomology class.  Consequently,
\begin{eqnarray*}
\lim _{j\rightarrow \infty}\big\langle\pi ^*(\beta_j\wedge \omega _Y^{l-p+p_0}),S\big\rangle\lesssim c_{p-q}(n)\|\pi ^*(\omega _Y^{l-q+p_0})\wedge
S\|=c_{p-q}(n)a_{p_0,q}(n).
\end{eqnarray*}
This completes the proof of the first assertion.

For the second assertion, when $q<p_0$ the form $\beta_j\wedge \omega _Y^{l-p+p_0}$ has bidegree 
$(l-q+p_0,l-q+p_0)$ which is bigger than $(l,l)$, thus must be $0$ since $Y$ has
dimension $l$. It follows that $T=0$ and the integral in the lemma is $0$ as well.
\endproof

\begin{lemma}\label{lemma_key}
 There exists a constant $A>0$ such that for all $0\leq p_0\leq k-l$, $p_0\leq p\leq l+p_0$ and all $n,r\geq 1$
$$ a_{p_0,p}(nr)
\leq  A^r \sum \prod_{s=1}^r a_{p_{s-1},p_s}(n)c_{p-p_s}(n),$$
where the  sum is taken over $(p_1,\ldots,p_r)$ with 
$p_{0}\leq p_1\leq  p_2\leq \cdots \leq p_r\leq p$ and $p_{r-1}\leq k-l.$ 
\end{lemma}
\proof
We proceed  by induction on $r$. Clearly, the lemma is true for
$r=1$. 
Suppose  the lemma for $r$, we need to prove  it for $r+1$. In
what follows, $\lesssim$ denotes an inequality up to a multiplicative constant which depends only on the geometry
of $X$ and $Y$. 

Define $T^{(r)}:=(f^{nr})^*(\omega^p)$.  This is an almost-smooth current on $X$. Therefore, we have
$$T^{(r+1)}= (f^n)^\bullet(T^{(r)}).$$
By Proposition \ref{prop_reg_bis} applied to  $T^{(r)}$,  we can find 
smooth positive closed $(p,p)$-forms $T_i^{(r)}$
converging  weakly to a positive closed current $\widetilde T^{(r)}\geq T^{(r)}$
such that
$$\alpha_{p-q}(T_i^{(r)})\lesssim \alpha_{p-q}(T^{(r)}) \lesssim a_{q,p}(nr)$$
for $\max\{0,p-l\}\leq q\leq \min\{p,k-l\}$.
Then, using again that proposition, we deduce that
$$\{T_i^{(r)}\}\lesssim \sum_{\max\{0,p-l\}\leq q \leq \min\{p,k-l\}}
a_{q,p}(nr)\{\pi^*(\omega_Y^{p-q})\}\smile \{\omega_X^q\}.$$

Finally, we obtain from the above discussion and Lemma \ref{lemma_Pi_wedge} that 
\begin{eqnarray*}
\lefteqn{a_{p_0,p}(n(r+1))   =   \big\langle T^{(r+1)},\pi^*(\omega_Y^{l-p+p_0})\wedge
\omega_X^{k-l-p_0}\big\rangle} \\
& \leq & \liminf_{i\rightarrow\infty} \big\langle 
(f^n)^*(T^{(r)}_i),  \pi^*(\omega_Y^{l-p+p_0})\wedge
\omega_X^{k-l-p_0}\big\rangle \\
& \lesssim & \sum_{\max\{0,p-l\}\leq q\leq \min\{p,k-l\}} a_{q,p}(nr)\big\langle 
(f^n)^*(\pi^*\omega_Y^{p-q}\wedge\omega_X^q),\pi^*(\omega_Y^{l-p+p_0})\wedge
\omega_X^{k-l-p_0}\big\rangle \\
& \lesssim & \sum_{p_0\leq q\leq\min\{p,k-l\}} a_{q,p}(nr) a_{p_0,q}(n)c_{p-q}(n).
\end{eqnarray*}
These estimates together with the induction hypothesis imply the result.
\endproof

\begin{proposition} \label{prop_half_th_main} We have
 $$d_p(f)\geq  \max_{\max\{0,p-k+l\}\leq j\leq
   \min\{p,l\}}d_j(g)d_{p-j}(f|\pi)$$
for $0\leq p\leq k$.
\end{proposition}
\proof
Since $\pi^*(\omega_Y^j)\wedge \omega_X^{p-j}$ is a smooth $(p,p)$-form, we have 
$$\big\|(f^n)^*(\pi^*(\omega_Y^j)\wedge \omega_X^{p-j})\big\|\lesssim \lambda_p(f^n).$$
So, by definition of dynamical degrees and equality (\ref{eq_a_pp}), 
it is enough to bound  $\|(f^n)^*(\pi^*(\omega_Y^j)\wedge \omega_X^{p-j})\|$ from below by a constant times 
$\lambda_j(g^n) a_{p-j,p-j}(n)$. 

Using the identity $\pi\circ
f^n=g^n\circ\pi$ and that 
$\pi^*(\omega_Y^{l-j})\wedge 
\omega_X^{k-l-p+j}$ is smooth, we obtain
\begin{eqnarray*}
\lefteqn{\|(f^n)^*(\pi^*\omega_Y^j\wedge \omega_X^{p-j})\|} \\ 
& \gtrsim & \big\langle  (f^n)^*(\pi^*\omega_Y^j\wedge \omega_X^{p-j}), \pi^*(\omega_Y^{l-j})\wedge 
\omega_X^{k-l-p+j}\big\rangle\\
& = & \big\langle  (f^n)^*\pi^*(\omega_Y^j)\zwedge (f^n)^*(\omega_X^{p-j}), \pi^*(\omega_Y^{l-j})\wedge 
\omega_X^{k-l-p+j}\big\rangle\\
& = & \big\|  (f^n)^*\pi^*(\omega_Y^j)\wedge \pi^*(\omega_Y^{l-j}) \zwedge (f^n)^*(\omega_X^{p-j})\wedge 
\omega_X^{k-l-p+j}\big\|\\
& = & \big\|  \pi^\bullet[(g^n)^*(\omega_Y^j)\wedge \omega_Y^{l-j}]\zwedge (f^n)^*(\omega_X^{p-j})\wedge \omega_X^{k-l-p+j}\big\|.
\end{eqnarray*}

Observe that $(g^n)^*(\omega_Y^j)\wedge \omega_Y^{l-j}$ is a positive measure of mass $\lambda_j(g^n)$. 
Using a simple argument on cohomology as in 
Lemma  3.2 in  \cite{DinhNguyen}, we show that the last expression  
is equal to $\lambda_j(g^n)$ times the mass of the restriction of $(f^n)^*(\omega_X^{p-j})\wedge 
\omega_X^{k-l-p+j}$ to a generic fiber of $\pi$. Therefore, it is also equal to
$$\lambda_j(g^n)\big\langle  \pi^*(\omega_Y^l), (f^n)^*(\omega_X^{p-j})\wedge 
\omega_X^{k-l-p+j}\big\rangle=\lambda_j(g^n)a_{p-j,p-j}(n),$$
where for simplicity we normalize $\omega_Y$ so that $\omega_Y^l$ is a probability measure.
This completes the proof of the proposition.
\endproof

\noindent{\bf  Proof of  Theorem  \ref{th_main}.}
By Proposition \ref{prop_half_th_main},  we  only need to show that
 $$d_p(f)\leq  \max_{\max\{0,p-k+l\}\leq j\leq
   \min\{p,l\}}d_j(g)d_{p-j}(f|\pi)$$
for $0\leq p\leq k$. To do   this  we  argue exactly  as in the proof of Proposition 4.6 in \cite{DinhNguyen}  using
identity  (\ref{eq_a_pp}),  Lemma  \ref{lemma_b_p}  and Lemma \ref{lemma_key}.
\hfill $\square$

\bigskip

In the  rest of the  paper we  prove the corollaries of  Theorem  \ref{th_main}.
  
\bigskip

\noindent{\bf  Proof of  Corollary \ref{cor_distinct_degree}.}
Using    Theorem  \ref{th_main}, we  proceed  as in the proof of Corollary 1.3 in \cite{DinhNguyen}. 
\hfill $\square$

\bigskip

\noindent
{\bf Proof of Corollary \ref{cor_kodaira}.} 
Using    Theorem  \ref{th_main}, we  argue  as in the proof of Corollary 1.4 in \cite{DinhNguyen}. 
\hfill $\square$

\bigskip

We  recall here briefly the  definition of  the Albanese fibration of  a compact  K\"{a}hler manifold $X.$
Let   $H^0(X,\Omega_X)$ be the complex vector space of all holomorphic $1$-forms on $X.$ Since $X$ is compact  K\"{a}hler,
these  forms  are closed. Therefore,    to any closed path $\gamma$  we  associate  the linear  form $$H^0(X,\Omega_X) \ni\varphi \mapsto \int_\gamma\varphi $$   which depends only  on the homology class of $[\gamma]\in H_1(X,\Z).$
This   correspondence   identifies  the  component  without torsion of $H_1(X,\Z)$  with a  co-compact lattice  $\Gamma$ of the dual space $H^0(X,\Omega_X)^*.$
The Albanese  variety $\Alb(X)$  of $X$    is, by defintion,    the complex  torus   $ H^0(X,\Omega_X)^* /  \Gamma.$

Fix  a base  point $x\in X.$
Let $y\in X $ and   $\varphi\in H^0(X,\Omega_X).$  
Then, for different  paths  $\gamma$ connecting $x$ to $y,$ the corresponding values of  $\int_\gamma\varphi $ are
always  equal  modulo the  values of  $\int_\delta\varphi$ for some closed path $\delta$.
Consequently, we obtain
a  holomorphic map $\alb:\  X\rightarrow \Alb(X)$   defined    by
$$
\alb(y):=\int_x^y \varphi,\qquad  \varphi\in H^0(X,\Omega_X),
$$
where the integration  is  taken  over  an arbitrary   path $\gamma$ connecting $x$ to $y.$ This is the Albanese map of $X.$

\bigskip
 
\noindent
{\bf Proof of Corollary \ref{cor_albanese}.} 
Let $Y:=\alb(X)$ be the image of the Albanese map.  
If $\varphi$ is a holomorphic 1-form, then $f^*(\varphi)$ is a holomorphic 1-form outside an analytic set of codimension $\geq 2$. By Hartogs' theorem, this form extends to a holomorphic 1-form on $X$. Therefore, $f$ induces a linear operator
$f^*$ from $H^0(X,\Omega_X)$ to itself. 

This operator induces      
a dominant meromorphic map $g$ on $Y$ such that $f$ is $\alb$-semi-conjugate  to $g.$    By    Corollary \ref{cor_distinct_degree}, the assumption on $f$ implies that  $g$ also
has  distinct consecutive dynamical degrees (since dynamical degrees are bi-meromorphic invariants, we can desingularize $Y$ if necessary).  By  Corollary  \ref{cor_kodaira}, the Kodaira dimension $\kappa_Y$ of $Y$
satisfies $\kappa_Y\leq 0.$

On the  other hand, by Corollary  10.6 in Ueno's book  \cite{Ueno},  $\kappa_Y\geq 0.$
Hence,  $\kappa_Y=0.$
 But  by  this corollary again, we have 
$Y=\alb(X).$ 
\hfill $\square$

\bigskip
 
\noindent
{\bf Proof of Corollary  \ref{cor_campana}.} 
Let $c_X:\ X\rightarrow C(X)$ be the  core  fibration  constructed  by Campana in  \cite{Campana1}.
By  the  proof of Theorem 6.1  in \cite{AmerikCampana},  $C(X)$ is  a projective  variety. Moreover,  there  exists a bi-meromorphic  map  $c_f:\  C(X)\rightarrow C(X)$  such that $c_X\circ f=c_f\circ c_X$
and      that $c_f^n=\id$ for some $n\geq 1.$  

A priori, $C(X)$ may be singular, but we can use a blow-up and assume that $C(X)$ is  a smooth projective 
manifold. Clearly, $d_j(c_f)=1$  for $0\leq j\leq  \dim C(X).$ By Corollary \ref{cor_distinct_degree}, it follows from  the assumption on $f$  that  $\dim C(X)=0$. Thus,  
$c_X$  is  a  constant map.  Since Theorem 3.3 in \cite{Campana1} says that  the generic  fibers of  $c_X$ are  special, so is  $X.$
\hfill $\square$

\noindent
T.-C. Dinh, UPMC Univ Paris 06, UMR 7586, Institut de
Math{\'e}matiques de Jussieu, 4 place Jussieu, F-75005 Paris,
France.\\ 
{\tt  dinh@math.jussieu.fr}, {\tt http://www.math.jussieu.fr/$\sim$dinh}

\medskip

\noindent
V.-A.  Nguy{\^e}n,
Vietnamese Academy  of Science  and  Technology,
Institute of Mathematics,
Department  of Analysis,
18  Hoang Quoc  Viet  Road, Cau Giay  District,
10307 Hanoi, Vietnam.\\
{\tt nvanh@math.ac.vn}

\noindent
{\sc Current  address:}
Math{\'e}matique-B{\^a}timent 425, UMR 8628, Universit{\'e} Paris-Sud,
91405 Orsay, France.\\
  {\tt VietAnh.Nguyen@math.u-psud.fr}, {\tt http://www.math.u-psud.fr/$\sim$vietanh}

\medskip

\noindent
T.-T. Truong,
Department of Mathematics, Indiana University Bloomington, Bloomington, IN 47405, USA.\\
{\tt truongt@indiana.edu}

\end{document}